\documentclass[onefignum,onetabnum]{siamart171218}
\usepackage{geometry}
\usepackage[utf8]{inputenc}
\usepackage{xcolor}
\usepackage[english]{babel}
\usepackage{amsfonts,amssymb,amstext}
\usepackage{mathtools}
\usepackage[shortlabels]{enumitem}
\usepackage[tight,nice]{units}
\usepackage{mathtools}
\usepackage{url}
\usepackage{subcaption}
\usepackage{graphicx}
\usepackage{cancel}

\newtheorem{thm}{Theorem}
\numberwithin{thm}{section}
\newtheorem{cor}[thm]{Corollary}
\newtheorem{lem}[thm]{Lemma}
\newtheorem{prop}[thm]{Proposition}

\newtheorem{rem}[thm]{Remark}
\newtheorem{ass}{Assumption}

\numberwithin{prob}{section}

\newcommand{\pwm}{f}

\newcommand{\nrmSpaceTime}{{L^p(\mathcal{I},\mathbb{R}^n)}}
\newcommand{\numteeth}{m}

\newcommand{\err}{\boldsymbol{\epsilon}}

\newcommand{\expon}{\mathrm{e}}

\newcommand{\bu}{\mbox{\boldmath$u$}}
\newcommand{\bv}{\mbox{\boldmath$v$}}

\newcommand{\bff}{\mbox{\boldmath$f$}}
\newcommand{\f}{\mbox{$f$}}

\newcommand{\bJJ}{\mbox{\boldmath$J$}}

\newcommand{\bU}{\mbox{\boldmath$U$}}
\newcommand{\bV}{\mbox{\boldmath$V$}}

\newcommand{\errtot}{\err}

\newcommand{\errffn}{\err_{{f,n}}}

\headers{Parareal with Discontinuous Sources}{Gander, Kulchytska-Ruchka, Niyonzima and Schöps}

\title{A New Parareal Algorithm for Problems with Discontinuous
	Sources}

\author{%
	Martin J. Gander\thanks{Section de Math\'{e}matiques, University of Geneva, 2-4 Rue du Li\`{e}vre,  CH-1211 Geneva, Switzerland.}
	\and
	Iryna Kulchytska-Ruchka\thanks{Institut für Theorie Elektromagnetischer Felder, 
		Technische Universität Darmstadt, Schlossgartenstrasse 8, D-64289 Darmstadt, Germany.}
	\and
	Innocent Niyonzima\thanks{Department of Civil Engineering and Engineering Mechanics, Columbia University, 500 West 120th Street, NY 10027 New York, USA.}
	\and
	Sebastian Schöps\footnotemark[2]}

\begin{document}
	
	\maketitle
	
	\begin{abstract}
		The Parareal algorithm allows to solve evolution problems exploiting
		parallelization in time. Its convergence and stability have been
		prove{d} under the assumption of regular (smooth) inputs. {We
			present and analyze here a new Parareal algorithm} for ordinary
		differential equations which involve discontinuous right-hand
		sides. Such situations occur in various applications, e.g., when an
		electric device is supplied with a pulse-width-modulated
		signal. {Our new Parareal algorithm} uses a smooth input for the
		coarse problem with reduced dynamics. {We derive} error estimates
		{that show} how the input reduction influences the overall convergence
		rate of the algorithm. {We support our theoretical results by}
		numerical experiments{, and also test our new Parareal algorithm in
			an eddy current simulation of} an induction machine.
		
	\end{abstract}
	
	\begin{keywords}
		Evolution problems, parallel-in-time solution, Parareal, ODEs with discontinuous inputs, convergence analysis
	\end{keywords}
	
	\begin{AMS}
		34A34, 34A36, 34A37, 65L20, 78M10
	\end{AMS}
	
	\section{Introduction} \label{section:introduction}
	
	Due to {the} increasing computational power of modern computer
	systems, scientists are nowadays able to solve complex physical
	problems, and parallel computers allow to reduce the time {to
	obtain the solution} further. The first and most {natural} approach
	{to solve evolution problems in parallel} is to perform parallel
	computations in space by domain decomposition, {see}
	\cite{Quarteroni_1999a, Toselli_2005a, Gander_2006a, Boubendir_2012a}
	{and references therein}. However, {when} space-parallelization is exploited up to saturation, {and more processors are still
		available,} parallel-in-time methods are considered to be a
	complementary approach to achieve further numerical speed-up{, see
		\cite{Gander_2015a} for an overview of such techniques}.
	
	The Parareal algorithm was introduced by Lions, Maday, and Turinici in
	\cite{Lions_2001aa}. It has become a powerful tool, which allows to
	solve time-dependent problems in a time-parallel fashion.  The method
	has been applied to a wide range of problems \cite{Nielsen_2012a}, in
	particular: linear and nonlinear parabolic problems \cite{Staff_2003a,
		LiuJiang_2012a}, molecular dynamics \cite{Baffico_2002a}, stochastic
	ordinary differential equations (ODEs) \cite{Bal_2003a,
		Engblom_2009a}, Navier-Stokes equations \cite{Trindade_2004a,
		Fischer_2005a}, quantum control problems \cite{Maday_2003a,
		Maday_2007a} and low-frequency problems in electrical engineering
	\cite{Schops_2018aa}.
	
	{The Parareal algorithm is based on a decomposition} of the time
	{domain of interest} into {non-overlapping} {time intervals}
	(e.g., one {time interval} per processor) and {the} parallel
	solution of the governing equation on each {time
		interval}. Exchange of information at synchronization points is
	based on the action of fine and coarse propagators. Starting from a
	prescribed initial {guess}, both operators solve the underlying
	problem over {each} time {interval} and return the solution at
	the end of the {time interval}. The fine propagator is accurate
	{and} computationally {expensive}. It can be, for example, a
	classical time integrator, which uses a very fine time
	discretization. On the other hand, the coarse propagator {is less
		accurate, but much less expensive than the fine propagator} (e.g.,
	via time stepping over a coarse partition). {The Parareal
		algorithm} corrects the {approximate solution iteratively} until
	convergence.
	
	{Several techniques} for reducing {the} computational cost of
	Parareal are discussed in \cite{Maday_2008a}. In particular, for the
	time domain solution of partial differential equations (PDEs), the use
	of a coarse {mesh also in space} for the coarse propagator is
	proposed. This approach can be used within a multiscale setting
	\cite{Astorino_2012a, Chouly_2014a} or with spatial averaging
	operators \cite{Barenblatt_1998a}. A second idea is to perform model
	order reduction (MOR) for {the} extraction of a coarse propagator
	from the fine problem. Further reduced order techniques, developed
	{in} \cite{Chen_2014a, He_2010a}, involve spatial MOR also for the
	coarse problem.
	
	{These ideas} help to reduce {the cost of Parareal} by
	simplifying the coarse model in space. In this paper we propose to use
	a simpler coarse problem with respect to {the} time variable,
	similar to \cite{Haut_2014a}, where Parareal was applied to PDEs which
	exhibit scale separation in time. Our method is {specific for}
	problems {involving} discontinuous or multirate excitations, e.g.,
	pulse-width-modulated signals (PWM){, an example of which is shown}
	in Figure~\ref{fig:fig_motor_voltages2},
	\begin{figure}[t]
		\centering
		\includegraphics[width=0.5\linewidth]{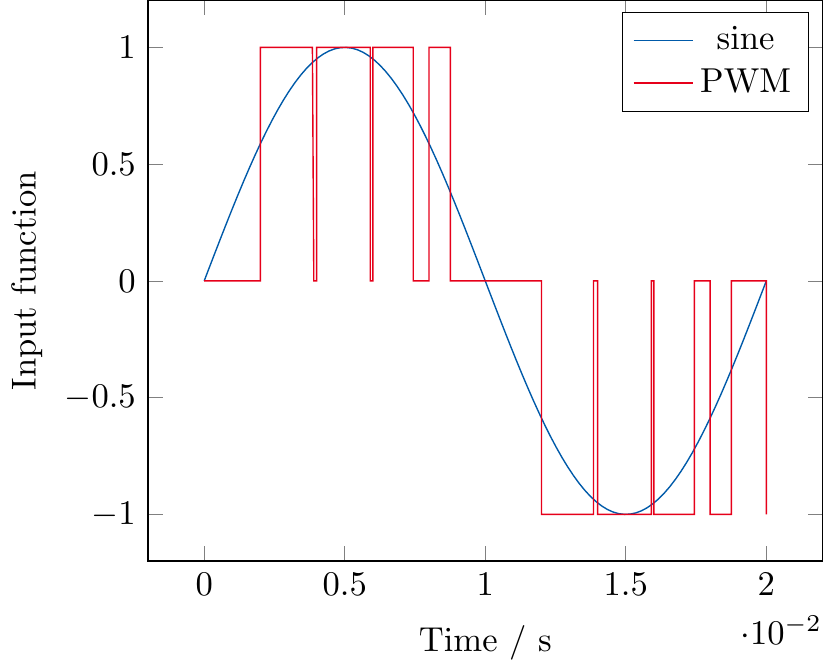}
		\caption{PWM signal with switching frequency of $f_\mathrm{s}=500$
			Hz, generating a sine wave of 50 Hz.}
		\label{fig:fig_motor_voltages2}
	\end{figure}
	or multiharmonic signals. Its main idea is to supply the coarse
	propagator with a smooth input, which features reduced dynamics, e.g.,
	a periodic waveform, which consists of the fundamental frequency
	only. For instance, in case of the PWM signal containing $10$ pulses
	on {the} time interval $[0, 0.02]\ s$, one could choose a sine wave
	of $50$ Hz to be the smooth input, {also shown} in
	Figure~\ref{fig:fig_motor_voltages2}. This allows the coarse
	propagator to use larger time steps and a high-order method.
	
	{Our} paper is organized as follows. The problem setting is
	described in Section~\ref{section:problem}. The {original} Parareal
	{algorithm} for a system of nonlinear ODEs, together with its error
	estimate from \cite{Gander_2008a} are recalled in
	Section~\ref{section:Parareal}.  {In}
	Section~\ref{section:Parareal_analysis_ode}{, we present our new
		Parareal algorithm for} a subclass of Carath\'{e}odory equations --
	equations, whose inputs may contain discontinuities with respect to
	the time variable{, and we derive a sharp convergence estimate
		using techniques developed in \cite{Gander_2008a}. We then measure
		the convergence rate of our new Parareal algorithm} numerically in
	Section~\ref{section:numerical_test} for an RL-circuit model{, and
		observe {a} very good agreement with our theoretical estimates. {In
			Section~\ref{Sec:Application}, we test the new Parareal algorithm
			applied to an eddy current} simulation of an induction machine. We
		finally present our conclusions} in
	Section~\ref{section:conclusions}.
	
	\section{Problem setting} \label{section:problem}
	
	{We} consider a nonlinear initial value problem (IVP) of
	non-auto\-nomous ODEs {of the form}
	\begin{align} 
	\bu^{\prime}(t) &= \bff(t, \bu(t)), \quad t \in \mathcal{I},\label{eq:section_Parareal_1_ODE}
	\\
	\bu(0) &= \bu_0,
	\label{eq:section_Parareal_1}
	\end{align}
	with right-hand side (RHS) $\bff: \mathcal{I} \times \mathbb{R}^n \to
	\mathbb{R}^n$ and solution $\bu: \mathcal{I} \to \mathbb{R}^n$ on the
	time interval $\mathcal{I} := (0, T]$. We are interested in problems
	for which the non-smooth (or even discontinuous) excitation can be
	separated from the smooth part of the RHS{, i.e.{,}}
	\begin{equation}
	{\bff}(t, \bu(t)) := \bar{\bff}(t, \bu(t)) + \tilde{\bff}(t),
	\label{eq:section_Parareal_3}
	\end{equation}
	where $\bar{\bff}(t, \bu(t))$ and $\tilde{\bff}(t)$ satisfy
	{the following two assumptions:}
	\begin{ass}\label{ass:smooth_mean_rhs}%
		{The} function $\bar{\bff}$ in \eqref{eq:section_Parareal_3}
		{is bounded and} sufficiently smooth in both arguments{, and
			it is} Lipschitz in the second argument {with
			Lipschitz constant} $L$.
	\end{ass}
	\begin{ass}\label{ass:small_perturb}%
		{The function} $\tilde{\bff}$ {in
			\eqref{eq:section_Parareal_3}} {belongs to}
		$\boldsymbol{L}^p(\mathcal{I},\mathbb{R}^n)$, $p\geq1,$ {with }{its norm given }{by $C_{p}:=\|\tilde{\bff}\|_\nrmSpaceTime$.}
	\end{ass}
	{Clearly,} the total RHS ${\bff}$ {has no continuity or
		smoothness} properties{, and therefore the Lindel\"{o}f} theory
	for existence and uniqueness of solution{s can not be applied} to
	\eqref{eq:section_Parareal_1_ODE}-\eqref{eq:section_Parareal_1}.
	However, one can {use the} solvability and uniqueness theory for
	Carath\'{e}odory equations, which can be found, e.g.,
	in~\cite{Filippov_2013a}.
	We recall that \eqref{eq:section_Parareal_1_ODE} is called a
	Carath\'{e}odory equation if its RHS $\bff(t,\bu)$ satisfies {the
		so called} Carath\'{e}odory conditions:
	\begin{enumerate}[a)]
		\item $\bff{(t,\bu)}$ is defined and continuous in $\bu$ for
		almost all $t$;
		\item $\bff{(t,\bu)}$ is measurable in $t$ for each $\bu$;
		\item $\|\bff(t, \bu)\| \leq m(t)$, with $m$ being a summable function on $\mathcal{I}$. 
	\end{enumerate}
	{It was proved in \cite{Filippov_2013a} that there exists a
		solution} to
	\eqref{eq:section_Parareal_1_ODE}-\eqref{eq:section_Parareal_1}, if
	$\bff(t,\bu)$ satisfies {the} Carath\'{e}odory conditions
	a)-c). Furthermore, if there exists a summable function $l(t)$
	s.t. $\forall (t,{\bv})$ and $\forall(t,\bu),$ with $t\in\mathcal{I}$
	\begin{equation}
	\label{uniqueness}
	\| \bff(t, \bu) - \bff(t, \bv) \|\leq l(t) \|\bu - \bv\|,
	\end{equation}
	then the solution is unique. {Note that}
	Assumptions~\ref{ass:smooth_mean_rhs} and \ref{ass:small_perturb}
	imply that {the} Carath\'{e}odory conditions and \eqref{uniqueness}
	are satisfied{, and hence} there exists a unique solution to
	\eqref{eq:section_Parareal_1_ODE}-\eqref{eq:section_Parareal_1}.
	
	\section{{Original} Parareal algorithm and convergence for smooth right-hand sides} \label{section:Parareal}
	
	We now {recall the original Parareal algorithm from
		\cite{Lions_2001aa} in the form described in \cite{Gander_2007a} for
		solving
		\eqref{eq:section_Parareal_1_ODE}-\eqref{eq:section_Parareal_1}.}
	The initial step of the algorithm consists {in} partitioning the
	time {domain} $(0,T]$ into {non-overlapping time} intervals
	$(T_{n-1}, T_{n}],$ $n = 1,\dots,N$ with $0 = T_0 < T_1 < T_2 <
	\ldots < T_N = T$. One can then define an evolution problem on
	each time {interval},
	\begin{align}
	&\bu_n^{\prime}(t) = \bff(t, \bu_n(t)), \quad t \in (T_{n-1}, T_{n}], \label{eq:section_Parareal_3_a_ODE}
	\\
	&\bu_n(T_{n-1}) = \bU_{n-1}
	\label{eq:section_Parareal_3_a}
	\end{align}
	for $n= 1,\dots,N$. {The initial} values $\bU_{n-1},$ $n =
	1,\dots,N$ need to be determined {such that the} solutions on each
	{time interval} $(T_{n-1}, T_{n}]$ coincide with {the}
	restriction of the solution of
	\eqref{eq:section_Parareal_1_ODE}-\eqref{eq:section_Parareal_1} to
	that {time interval}. {The Parareal algorithm computes by
		iteration better and better approximations of} these initial
	conditions: for {a given initial guess ${\bU_n^{{(0)}}}$,
		$n=0,\dots,N$, it solves for $k=0,1,\dots,K$}
	\begin{align}
	{{\bU_0^{(k+1)}}}&=\bu_0,\label{eq:section_Parareal_3_b_init}\\
	{{\bU_n^{(k+1)}}}&=\mathcal{F}\big(T_n, T_{n-1},{\bU^{(k)}_{n-1}}\big)
	+\mathcal{G}\big(T_n, T_{n-1},{\bU^{(k+1)}_{n-1}}\big) - \mathcal{G}\big(T_n, T_{n-1},\bU^{(k)}_{n-1}\big),{\quad n=1,\dots,N}.
	\label{eq:section_Parareal_3_b}
	\end{align}
	In \eqref{eq:section_Parareal_3_b} we denote by
	$\mathcal{F}(t,T_{n-1},\bU_{n-1})$ and
	$\mathcal{G}(t,T_{n-1},\bU_{n-1})$ the numerical solution
	{propagators} of the IVP
	\eqref{eq:section_Parareal_3_a_ODE}-\eqref{eq:section_Parareal_3_a}.
	Both of them propagate {the} initial value $\bU_{n-1}$ in time on
	$(T_{n-1},T_{n}]${, but they} differ in accuracy: {the fine
		propagator} $\mathcal{F}$ gives a very accurate, {but}
	expensive {approximate solution to} the IVP, whereas {the
		coarse propagator} $\mathcal{G}$ {gives an inexpensive, but}
	less accurate solution. The first term of the RHS in
	\eqref{eq:section_Parareal_3_b} involves quantities, which are
	already known at the iteration $k+1$ and, therefore, can be computed
	in parallel. The last one is known as well, since it has been
	already computed at the previous iteration. The term
	$\mathcal{G}\big(T_n, T_{n-1},{\bU^{(k+1)}_{n-1}}\big)$ involves the
	approximation $\bU_{n-1}^{(k+1)}$, $n=1,\dots,N$ which has not
	{yet been obtained at} the beginning of {the} iteration
	$k+1$. Therefore, its calculation cannot be parallelized and the
	{coarse but inexpensive} propagator $\mathcal{G}$ is applied
	sequentially.
	
	{We now state} the convergence result for problems with smooth RHS
	$\bff$, {which was proved in \cite{Gander_2008a} under the
		assumption that} each {time interval} has the same length $\Delta
	T = T/N$.
	\begin{thm}\label{thm:section_Parareal_3}
		Let the RHS $\bff$ be smooth enough and assume {that}
		$\mathcal{F}\big(T_n, T_{n-1},\bU^{(k)}_{n-1}\big)$ is the exact
		solution to
		\eqref{eq:section_Parareal_3_a_ODE}-\eqref{eq:section_Parareal_3_a}
		at $T_n$ with initial value $\bU^{(k)}_{n-1}.$ Furthermore,
		\begin{itemize}
			\item let $\mathcal{G}\big(T_n, T_{n-1},\mathbf{U}^{(k)}_{n-1}\big)$
			be an approximate solution with local truncation error bounded by
			$C_3\Delta T^{p+1},$ which can be expanded for $\Delta T$ small as
			\begin{equation}\label{eq:section_Parareal_4_a}
			\mathcal{F}(T_n, T_{n-1}, \bU) - \mathcal{G}(T_n, T_{n-1}, \bU) = 
			c_{l+1}(\bU) \Delta T^{l+1} + c_{l+2}(\bU) \Delta T^{l+2} +\dots
			\end{equation}
			with {an initial value $\bU$ and} continuously
			differentiable functions $c_{i},\ i=l+1,l+2,\dots$;
			\item assume {that} $\mathcal{G}$ satisfies the Lipschitz
			condition
			\begin{equation}\label{eq:section_Parareal_4_b}
			\|\mathcal{G}\big(t+\Delta T, t, \bU\big)-\mathcal{G}\big(t+\Delta T, t, \bV\big)\|\leq(1+C_2\Delta T)\|\bU-\bV\|
			\end{equation}
			for $t\in\mathcal{I}$ and for all $\bU, \bV$, with constant
			$C_2$.
		\end{itemize}
		Then at iteration $k$ of the Parareal algorithm
		\eqref{eq:section_Parareal_3_b_init}-\eqref{eq:section_Parareal_3_b}
		we have the {error} bound
		\begin{equation}\label{eq:section_Parareal_5}
		||\bu(T_n)-\bU_n^{(k)}|| 
		\leq 
		\frac{C_3}{C_1} \frac{(C_1 \Delta T^{l+1})^{k+1}}{(k+1)!}(1+C_2\Delta T)^{n-k-1}\prod_{j=0}^k (n-j),
		\end{equation}
		where the constant $C_1$ comes from the expansion
		\eqref{eq:section_Parareal_4_a} and {the} Lipschitz continuity of
		$c_{i}$, $i=~l+1,l+2,\dots${, see the proof in \cite{Gander_2008a}.}
	\end{thm} 
	
	\section{{A new Parareal algorithm for} non-smooth sources}\label{section:Parareal_analysis_ode}
	
	We now omit the assumption of smoothness {on} the RHS and allow
	discontinuities in the time-dependent input $\tilde{\bff}${, 
		considering the} IVP
	\eqref{eq:section_Parareal_1_ODE}-\eqref{eq:section_Parareal_1} with
	$\bff$ as in \eqref{eq:section_Parareal_3} such that {only}
	Assumptions~\ref{ass:smooth_mean_rhs} and \ref{ass:small_perturb} are
	{satisfied}.
	
	When one deals with a highly oscillatory or discontinuous source, the
	coarse propagator $\mathcal{G}$ might not capture its dynamics if low
	accuracy, i.e., big time steps are used. {This} may lead to solving
	a coarse problem, which does not contain enough information about the
	original input{, and} it is not clear how {this influences} the
	overall convergence of the Parareal algorithm. For this reason, we
	propose to define a smooth input, which is appropriate for coarse
	discretization. {Therefore, in our new Parareal algorithm, the
		coarse propagator solves} the modified problem with reduced dynamics
	\begin{align} 
	\bu^{\prime}(t) &= \bar{\bff}(t, \bu(t)), \quad t \in \mathcal{I},\label{eq:section_Parareal_6_ODE}
	\\
	\bu(0) &= \bu_0,
	\label{eq:section_Parareal_6}
	\end{align}
	while the fine propagator
	$\mathcal{F}$ is still applied to the original problem
	\eqref{eq:section_Parareal_1_ODE}-\eqref{eq:section_Parareal_1}. In
	particular, the coarse propagator $\bar{\mathcal{G}}$ on the time
	{interval} $(T_{n-1}, T_{n}]$ for $n=1,\dots,N$ solves
	\begin{align}
	&\bu_n^{\prime}(t) = \bar{\bff}(t, \bu_n(t)), \quad t \in (T_{n-1}, T_{n}], \label{eq:section_Parareal_3_a_reduced_ODE}
	\\
	&\bu_n(T_{n-1}) = \bU_{n-1}.
	\label{eq:section_Parareal_3_a_reduced}
	\end{align}
	Our {new Parareal algorithm then computes} for $k=0,1,\dots,K$ and
	$\ n=1,\dots,N$
	\begin{align}
	{{\bU_0^{(k+1)}}}&
	=\bu_0,\label{eq:section_Parareal_3_b_init_reduced}\\
	{{\bU_n^{(k+1)}}}&=\mathcal{F}\big(T_n, T_{n-1},{\bU^{(k)}_{n-1}}\big)
	+\bar{\mathcal{G}}\big(T_n, T_{n-1},{\bU^{(k+1)}_{n-1}}\big) - \bar{\mathcal{G}}\big(T_n, T_{n-1},\bU^{(k)}_{n-1}\big).
	\label{eq:section_Parareal_3_b_reduced}
	\end{align}
	The initial approximation can be calculated {using the coarse
		propagator},
	\begin{equation}
	\label{init_approx_reduced}
	\bU_n^{(0)}:=\bar{\mathcal{G}}\big(T_n, T_{n-1},{\bU^{(0)}_{n-1}}\big),\quad n=1,\dots,N.
	\end{equation}
	For a given initial value $\bU$, we define the difference between
	{the} exact solution of \eqref{eq:section_Parareal_3_a_ODE} and the
	numerical solution of the reduced coarse problem
	\eqref{eq:section_Parareal_3_a_reduced_ODE} as
	\begin{equation}\label{eq:section_Parareal_7}
	\errtot_n(T_n,\bU){:}=\mathcal{F}(T_n, T_{n-1}, \bU) - \bar{\mathcal{G}}(T_n, T_{n-1}, \bU).
	\end{equation}
	{For analysis purposes, we also} introduce an additional propagator
	$\bar{\mathcal{F}}$, which, as $\bar{\mathcal{G}}$, solves
	\eqref{eq:section_Parareal_6_ODE}-\eqref{eq:section_Parareal_6}, but
	is exact. {We can then} express the error $\errtot_n$ as
	\begin{align}\label{errff}
	\errtot_{n}(T_n,\bU)&=\underbrace{\mathcal{F}(T_n, T_{n-1}, \bU) - \bar{\mathcal{F}}(T_n, T_{n-1}, \bU)}_{=:\errffn(T_n)}
	+ \bar{\mathcal{F}}(T_n, T_{n-1}, \bU) - \bar{\mathcal{G}}(T_n, T_{n-1}, \bU).
	\end{align}
	We {now show} that the error $\errffn$ between the
	solution of the original ODE \eqref{eq:section_Parareal_3_a_ODE} and
	{the solution} of the reduced {ODE}
	\eqref{eq:section_Parareal_3_a_reduced_ODE} with initial value $\bU$
	at $T_{n-1}$ does not depend on $\bU$.
	\begin{prop}
		{If} Assumptions~\ref{ass:smooth_mean_rhs} and
		\ref{ass:small_perturb} hold, then the error $\errffn$ {from
			\eqref{errff} solves} the IVP
		\begin{equation}\label{eq:section_analysis_Parareal_10}
		\begin{aligned}
		&\errffn^{\prime}(t) = \bJJ(t, \errffn(t))\errffn(t) + \tilde{\bff}(t), \quad t \in (T_{n-1},T_n],\\  
		&\errffn(T_{n-1}) = 0,
		\end{aligned}
		\end{equation}
		where $\bJJ(t, \errffn(t))$ is defined in \cite{Dahlquist_2007aa} as
		the neighborhood average of the Jacobian, given by
		\begin{equation}\label{eq:section_analysis_Parareal_11}
		\bJJ(t, \errffn(t)) = \int_{0}^{1} \displaystyle \frac{\partial \bar{\bff}}{\partial \bu} 
		\left( t, \bar{\bu}(t) + \theta \errffn(t)\right) \mathrm{d}\theta.
		\end{equation}
	\end{prop}
	\begin{proof}
		Let $\bu_{n}$ and $\bar{\bu}_{n}$ solve
		\eqref{eq:section_Parareal_3_a_ODE}-\eqref{eq:section_Parareal_3_a}
		and
		\eqref{eq:section_Parareal_3_a_reduced_ODE}-\eqref{eq:section_Parareal_3_a_reduced},
		respectively. The error $\errffn$ on $[T_{n-1},T_n]$ is then defined
		as the difference $\errffn{:}=\bu_{n}-\bar{\bu}_{n}$.  Subtracting
		equation \eqref{eq:section_Parareal_3_a_reduced_ODE} from
		\eqref{eq:section_Parareal_3_a_ODE} and initial condition
		\eqref{eq:section_Parareal_3_a_reduced} from
		\eqref{eq:section_Parareal_3_a} we obtain
		\begin{equation}\label{eq:section_analysis_Parareal_10_intermediate}
		\begin{aligned}
		&\errffn^{\prime}(t) = \bar{\bff} \left(t, \bar{\bu}_{n}(t) + \errffn(t)\right) - \bar{\bff} \left( t, \bar{\bu}_{n}(t)\right) + \tilde{\bff}(t), \quad t \in (T_{n-1},T_n],
		\\  
		&\errffn(T_{n-1}) = 0.
		\end{aligned}
		\end{equation}
		Using the {fundamental theorem of calculus, we get}
		\begin{align}
		&\bar{\bff} \left( t, \bar{\bu}_{n}(t) + \errffn(t)\right) - \bar{\bff} \left(t, \bar{\bu}_{n}(t)\right)
		= \int_{0}^{1} \displaystyle \frac{\partial \bar{\bff}}{\partial \theta} \left( t, \bar{\bu}_{n}(t) + \theta \errffn(t)\right) \mathrm{d}\theta \\
		&= \int_{0}^{1} \displaystyle \frac{\partial \bar{\bff}}{\partial \bu} \left( t, \bar{\bu}_{n}(t) + \theta \errffn(t)\right) \errffn(t) \mathrm{d}\theta =: \bJJ(t, \errffn(t))\errffn(t),
		\label{eq:section_analysis_Parareal_12}
		\end{align}
		which leads to \eqref{eq:section_analysis_Parareal_10}.
	\end{proof}
	\begin{rem}
		We note that the IVP \eqref{eq:section_analysis_Parareal_10} is
		again well-defined in the sense of Carath\'{e}odory theory.
	\end{rem}
	In the following lemma we derive {a} bound for the error
	$\errffn(T_n),$ solution to \eqref{eq:section_analysis_Parareal_10}.
	\begin{lem}\label{lemma:bound_errfn}
		Let Assumptions~\ref{ass:smooth_mean_rhs} and
		\ref{ass:small_perturb} hold, and let the {time interval} length
		$\Delta T = T/N$ be small. Then there exists $C_4>0$ s.t. the
		solution to \eqref{eq:section_analysis_Parareal_10} can be bounded
		at $T_{n}$ {by}
		\begin{equation}\label{bound_errffn+DeltaT}
		\|\errffn(T_n)\|\leq C_4C_{p} \Delta T^{1/q},
		\end{equation}
		where the integer $q\geq1$ is defined {by the relation}
		$1/p+1/q=1${, and
			{$C_{p}:=\|\tilde{\bff}\|_{\nrmSpaceTime}$} %
			{is} from Assumption \ref{ass:small_perturb}.}
	\end{lem}
	\begin{proof}
		{Let us denote an }{arbitrary spatial }{norm of $\tilde{\bff}(t)$ in $\mathbb{R}^n$ by $\varepsilon(t):=\|\tilde{\bff}(t)\|$.}
		{Then,} based on Theorem~10.2 in \cite{Hairer_2000a} for $t\geq T_{n-1}$,
		{one can bound the error $\errffn$ by}
		\begin{equation}\label{bound_errffn_Lp}
		\|\errffn(t)\|\leq \expon^{L(t-T_{n-1})}\int_{T_{n-1}}^{t}\expon^{-L(s-T_{n-1})}\varepsilon(s)\mathrm{d}s,
		\end{equation}
		{since initially} at $T_{n-1}$ the error $\errffn$ {equals zero
			and is thus bounded, the} norm $\|\tilde{\bff}(t)\|$ is
		bounded by $\varepsilon(t),$ %
		and {the} function $\bar{\bff}$ is Lipschitz continuous with
		Lipschitz constant $L$, as stated in
		Assumption~\ref{ass:smooth_mean_rhs}. Taking $t=T_{n}$ in
		\eqref{bound_errffn_Lp} and using H\"{o}lder's inequality together
		with {a} Taylor expansion for $\Delta T$ {small, we obtain}
		\begin{align}
		\|\errffn(T_n)\| &\leq \expon^{L\Delta T}\int_{T_{n-1}}^{T_n}\left|\expon^{-L(s-T_{n-1})}\varepsilon(s)\right|\mathrm{d}s\nonumber\\
		&\leq \expon^{L\Delta T}\left(\int_{T_{n-1}}^{T_n}\left|\expon^{-L(s-T_{n-1})}\right|^q\mathrm{d}s\right)^{1/q} \left(\int_{T_{n-1}}^{T_n}|\varepsilon(s)|^p\mathrm{d}s\right)^{1/p}\nonumber\\
		&= \left(1+L\Delta T+\mathcal{O}\left(\Delta T^2\right)\right)\left[\Delta T+\mathcal{O}\left(\Delta T^2\right)\right]^{1/q}\|\varepsilon\|_{L^p(T_{n-1},T_n)}\nonumber\\
		&\leq C_{p}\Delta T^{1/q}+\mathcal{O}\left(\Delta T^{2/q}\right)\leq C_4 C_{p}\Delta T^{1/q}, \nonumber
		\end{align}
		with $q\geq1$ satisfying $1/p+1/q=1$ and the constant $C_4$ coming
		from the definition of the Landau symbol "big $\mathcal{O}$".
	\end{proof}
	
	We {can now prove a convergence result for our new Parareal
		algorithm for non-smooth input}
	\eqref{eq:section_Parareal_3_b_init_reduced}-\eqref{eq:section_Parareal_3_b_reduced}
	for problem
	\eqref{eq:section_Parareal_1_ODE}-\eqref{eq:section_Parareal_1},
	{which is} similar to {that of} Theorem~\ref{thm:section_Parareal_3}{, derived for the} %
	case of smooth RHS. {Like in Theorem
		~\ref{thm:section_Parareal_3}, we also assume that the time
		intervals have equal} length, $\Delta T = T/N$.
	\begin{thm}\label{thm:section_Parareal_5}
		Let Assumptions~\ref{ass:smooth_mean_rhs} and \ref{ass:small_perturb}
		be {satisfied, and assume that} $\mathcal{F}\big(T_n,
		T_{n-1},\bU^{(k)}_{n-1}\big)$ is the exact solution to
		\eqref{eq:section_Parareal_3_a_ODE}-\eqref{eq:section_Parareal_3_a} at
		$T_n$ with initial value $\bU^{(k)}_{n-1}.$ Furthermore,
		\begin{itemize}
			\item let $\bar{\mathcal{G}}\big(T_n,
			T_{n-1},\mathbf{U}^{(k)}_{n-1}\big)$ be an approximate solution to
			\eqref{eq:section_Parareal_3_a_reduced_ODE}-\eqref{eq:section_Parareal_3_a_reduced}
			with local truncation error bounded by $\bar{C_3}\Delta T^{l+1},$
			which can be expanded for $\Delta T$ small as
			\begin{equation}\label{eq:section_Parareal_4_a_reduced}
			\bar{\mathcal{F}}(T_n, T_{n-1}, \bU) - \bar{\mathcal{G}}(T_n, T_{n-1}, \bU) = 
			\bar{c}_{l+1}(\bU) \Delta T^{l+1} + \bar{c}_{l+2}(\bU) \Delta T^{l+2} +\dots
			\end{equation}
			with continuously differentiable functions $\bar{c}_{i},\ i=l+1,l+2,\dots,$ {and where $\bar{\mathcal{F}}\big(T_n,
				T_{n-1},\bU\big)$ denotes} the exact solution to
			\eqref{eq:section_Parareal_3_a_reduced_ODE} at $T_n,$ starting from
			{the} initial value $\bU;$
			\item assume $\bar{\mathcal{G}}$ satisfies the Lipschitz condition
			\begin{equation}\label{eq:section_Parareal_4_b_reduced}
			\|\bar{\mathcal{G}}\big(t+\Delta T, t, \bU\big)-\bar{\mathcal{G}}\big(t+\Delta T, t, \bV\big)\|\leq(1+C_2\Delta T)\|\bU-\bV\|
			\end{equation}
			for $t\in\mathcal{I}$ and for all $\bU$, $\bV$.
		\end{itemize}
		Then at iteration $k${, the new Parareal algorithm
			\eqref{eq:section_Parareal_3_b_init_reduced}-\eqref{eq:section_Parareal_3_b_reduced}
			satisfies the error} bound
		\begin{equation}\label{eq:section_Parareal_5_b}
		||\bu(T_n) - \bU_n^k|| \leq 
		\bar{C}_{1}^k\left[C_4C_{p}{\Delta T}^{(l+1)k+1/q}+\bar{C}_3 \left( {\Delta T}^{l+1} \right)^{k+1}\right]
		\frac{(1 + C_2 \Delta T)^{n-k-1}}{(k+1)!}\prod_{j = 0}^{k}(n - j)
		\end{equation}
		with the integer $q\geq1$ {defined by the relation} $1/p+1/q=1$, {constants $C_{p}$ and $C_4$ from Lemma~\ref{lemma:bound_errfn}, and $\bar{C_1}>0$ determined by the Lipschitz constant of $\bar{c}_{l+1}$ and the expansion \eqref{eq:section_Parareal_4_a_reduced}}.
	\end{thm}
	\begin{proof}
		{By adding and subtracting the same terms, we obtain from the
			new} Parareal update formula {for the error} {of}
		\eqref{eq:section_Parareal_3_b_reduced}
		\begin{align}
		\bu(T_n) - \bU_n^{(k+1)} &=
		\mathcal{F}\left(T_n, T_{n-1},\bu(T_{n-1})\right)
		-\mathcal{F}\left(T_n, T_{n-1},\bU_{n-1}^{(k)}\right)
		\nonumber\\
		&+\bar{\mathcal{G}}\left(T_n, T_{n-1},\bU_{n-1}^{(k)}\right)
		-\bar{\mathcal{G}}\left(T_n, T_{n-1},\bU_{n-1}^{(k+1)}\right)
		\nonumber\\
		&\pm\bar{\mathcal{F}}\left(T_n, T_{n-1},\bu(T_{n-1})\right)
		\pm\bar{\mathcal{G}}\left(T_n, T_{n-1},\bu(T_{n-1})\right)
		\pm\bar{\mathcal{F}}\left(T_n, T_{n-1},\bU_{n-1}^{(k)}\right)
		\nonumber\\
		&=\underbrace{\mathcal{F}\left(T_n, T_{n-1}, \bu(T_{n-1})\right)
			-\bar{\mathcal{F}}\left(T_n, T_{n-1},\bu(T_{n-1})\right)}_{=\errffn(T_n)}
		\nonumber\\
		&+\underbrace{\bar{\mathcal{F}}\left(T_n, T_{n-1},\bu(T_{n-1})\right)
			-\bar{\mathcal{G}}\left(T_n, T_{n-1},\bu(T_{n-1})\right)}_{=\bar{c}_{l+1}\left(\bu(T_{n-1})\right)\Delta T^{l+1}+\dots}
		\nonumber\\
		&-\underbrace{\bigg(\mathcal{F}\left(T_n, T_{n-1},\bU_{n-1}^{(k)}\right)
			-\bar{\mathcal{F}}\left(T_n, T_{n-1},\bU_{n-1}^{(k)}\right)\bigg)}_{=\errffn(T_n)}
		\nonumber\\
		&- \underbrace{\bigg(\bar{\mathcal{F}}\left(T_n, T_{n-1},\bU_{n-1}^{(k)}\right)
			- \bar{\mathcal{G}}\left(T_n, T_{n-1},\bU_{n-1}^{(k)}\right)\bigg)}_{=\bar{c}_{l+1}\left(\bU_{n-1}^{(k)}\right)\Delta T^{l+1}+\dots}
		\nonumber\\
		&+ \bar{\mathcal{G}}\left(T_n, T_{n-1},\bu(T_{n-1})\right)
		- \bar{\mathcal{G}}\left(T_n, T_{n-1},\bU_{n-1}^{(k+1)}\right).
		\end{align}
		Using {the} Lipschitz continuity of $\bar{c}_{l+1}$ and {the}
		Lipschitz condition \eqref{eq:section_Parareal_4_b_reduced}, we obtain
		the bound
		$$
		\|\bu(T_n) - \bU_n^{(k+1)}\|\leq \bar{C_1}\Delta T^{l+1}\|\bu(T_n) - \bU_{n-1}^{(k)}\|+(1+C_2\Delta T)\|\bu(T_n) - \bU_{n-1}^{(k+1)}\|
		$$
		{with a positive constant $\bar{C_1}$.}
		{In order to obtain a bound on the error, we now} consider the
		{corresponding} recurrence relation $e_{n}^{k+1}=\alpha
		e_{n-1}^{k}+\beta e_{n-1}^{k+1}$ with $\alpha=\bar{C_1}\Delta T^{l+1}$
		and $\beta=1+C_2\Delta T$. {Due to the initial guess from the
			coarse propagator} \eqref{init_approx_reduced}, the initial error
		{can be estimated for $n=1,\dots,N$ by}
		\begin{align}
		\|\bu(T_n) - \bU_n^{(0)}\|&=\|\mathcal{F}\left(T_n, T_{n-1}, \bu(T_{n-1})\right)-\bar{\mathcal{G}}\left(T_n, T_{n-1},\bU_{n-1}^{(0)}\right)\|\nonumber\\
		&\leq \|\mathcal{F}\left(T_n, T_{n-1}, \bu(T_{n-1})\right)-\bar{\mathcal{G}}\left(T_n, T_{n-1},\bu(T_{n-1})\right)\|\nonumber\\
		&+\|\bar{\mathcal{G}}\left(T_n, T_{n-1},\bu(T_{n-1})\right)-\bar{\mathcal{G}}\left(T_n, T_{n-1},\bU_{n-1}^{(0)}\right)\|\nonumber\\
		&\leq \|\mathcal{F}\left(T_n, T_{n-1}, \bu(T_{n-1})\right)-\bar{\mathcal{F}}\left(T_n, T_{n-1},\bu(T_{n-1})\right)\|\nonumber\\
		&+\bar{C_3}\Delta T^{l+1} + (1+C_2\Delta T)\|\bu(T_n) - \bU_{n-1}^{(0)}\|.\nonumber
		\end{align}
		{Now} Lemma~\ref{lemma:bound_errfn} gives us {a} bound for the
		first term {on the right-hand side} above{, and we thus obtain
			for the bounding initial recurrence relation}
		\begin{equation}\label{gamma}
		e_{n}^{0}=\gamma+\beta e_{n-1}^{0},\quad \gamma:=C_4C_{p}\Delta T^{1/q}+\bar{C_3}\Delta T^{l+1}.
		\end{equation}
		{We can now follow the same reasoning} as in \cite{Gander_2008a}
		{to obtain} the estimate \eqref{eq:section_Parareal_5_b}.
	\end{proof}
	\begin{cor}\label{cor_estimate}
		Let {the} assumptions of Theorem~\ref{thm:section_Parareal_5}
		be satisfied. If {$\tilde{\bff}\in L^\infty(\mathcal{I},\mathbb{R}^n)$}  %
		in Assumption~\ref{ass:small_perturb}, then the estimate
		\eqref{eq:section_Parareal_5_b} {becomes}
		\begin{equation}\label{eq:section_Parareal_5_b_infty}
		||\bu(T_n) - \bU_n^k|| \leq
		\bar{C}_{1}^k\left[C_4C_{\infty} {\Delta T}^{(l+1)k+1}+\bar{C}_3 \left( {\Delta T}^{l+1} \right)^{k+1}\right]
		\frac{(1 + C_2 \Delta T)^{n-k-1}}{(k+1)!}\prod_{j = 0}^{k}(n - j).
		\end{equation}
	\end{cor}
	\begin{proof}
		{Using Theorem 10.2 from \cite{Hairer_2000a} and}  %
		{boundedness of the vector norm $\|\tilde{\bff}(t)\|\leq C_{\infty}$} on $\mathcal{I}${, we
			obtain the} bound
		\begin{equation}\label{bound_errffn}
		\|\errffn(t)\|\leq \frac{C_{\infty}}{L}\left(\expon^{L(t-T_{n-1})}-1\right),\quad t\leq T_{n-1}.
		\end{equation}
		For small $\Delta T$ {this} implies that there exists $C_4>0$ s.t.
		\begin{equation}\label{bound_errffn_Linfty}
		\|\errffn(T_n)\|\leq C_4C_{\infty}\Delta T,
		\end{equation}
		{and} following the proof of
		Theorem~\ref{thm:section_Parareal_5}{, we obtain} the estimate
		\eqref{eq:section_Parareal_5_b_infty}.
	\end{proof}
	\begin{rem}
		{From the convergence estimate \eqref{eq:section_Parareal_5_b}, we
			see that if} the norm $\|\tilde{\bff}\|_\nrmSpaceTime$
		in Assumption~\ref{ass:small_perturb} is small enough, then the second
		term in the estimate \eqref{eq:section_Parareal_5_b} will dominate
		{initially, and the convergence rate will be as} for the {original}
		Parareal algorithm, where coarse and fine propagators both solve the
		same problem. {This explains the key innovation in our new Parareal
			algorithm, namely to use a {suitable} smooth input $\bar{\bff}$ for our new
			coarse} propagator $\bar{\mathcal{G}}$, {in order to avoid} a
		considerable reduction of the Parareal convergence order.
	\end{rem}
	
	\section{Numerical {experiments for a model problem}}
	\label{section:numerical_test}
	
	{We now compare the performance of our new Parareal algorithm to
		the one of the original Parareal algorithm, and test the accuracy of
		our error estimates on the model of the RL-circuit shown} in
	Figure~\ref{RL_circuit_fig}.
	\begin{figure}
		\centering
		\includegraphics[width=0.4\linewidth]{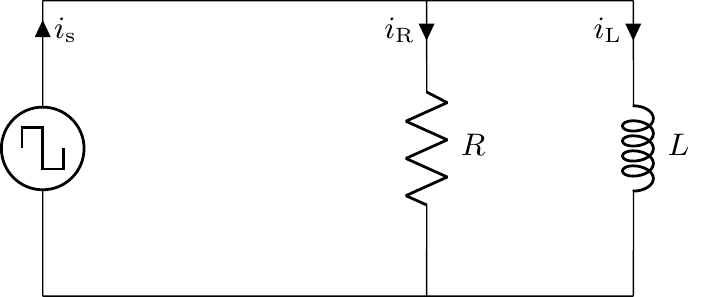}
		\caption{RL-circuit model.}
		\label{RL_circuit_fig}
	\end{figure}
	{The equations for this circuit are}
	\begin{equation}\label{RL_circuit_model}
	\begin{aligned} 
	\frac{1}{R}\phi^\prime(t)+\frac{1}{L}\phi(t)&= \pwm_{\numteeth{}}\left(t\right),\quad t\in(0,T],
	\\
	{\phi}(0) &= 0,
	\end{aligned} 
	\end{equation}
	where $R=0.01$\;$\Omega$ is the resistance, $L=0.001$\;H denotes the
	inductivity, $T=0.02$\;s is the period, and $\pwm_{\numteeth{}}$
	{is the} supplied PWM current source (in A) with $\numteeth{}$
	denoting {the} number of pulses, i.e.,
	\begin{equation}\label{input_pwm}
	\pwm_{\numteeth}(t)=\begin{cases}
	\mathrm{sign}\left[\sin\left(\dfrac{2\pi}{T} t\right)\right],\ & s_m(t)-\left|\sin\left(\dfrac{2\pi}{T}t\right)\right|<0,\\
	0,\ & \mathrm{otherwise,}
	\end{cases}
	\end{equation}
	where $s_m(t)=\dfrac{\numteeth}{T}t -
	\left\lfloor\dfrac{\numteeth}{T}t\right\rfloor,$ $t\in[0,T]$ is the
	common sawtooth pattern. In Figure~\ref{fig:fig_motor_voltages2}
	{we showed already} the PWM of switching frequency
	$f_s=\numteeth{}/T=500$ Hz, which consists of $\numteeth{}=10$
	pulses. {Note that the} values, which the depicted PWM signal
	attains, are {only} $-1,0,1$. Our numerical tests deal with
	{the} base frequency of 50 Hz and a modulation of $20$ kHz
	($m=400$), which is practically relevant in many applications in
	electrical engineering.
	
	\subsection{{Performance of the original Parareal algorithm}}
	
	{In the original Parareal algorithm,} both {the} fine and
	{the} coarse problem {use} the PWM signal \eqref{input_pwm}. The
	coarse propagator on each time {interval} is chosen to be the
	Backward Euler ({BE}) method of order $l=1$. For {a small number
		of processors, $N\ll m$, the coarse propagator will} not resolve the
	dynamics of the excitation, and therefore {the original
		convergence} arguments are not applicable:
	Theorem~\ref{thm:section_Parareal_3} is valid {only for $N$ large
		enough}, when the coarse propagator resolves all the pulses and the
	function is locally smooth{, and only in} this regime, one can
	expect that the high convergence rate {of the original Parareal
		algorithm} is maintained.  {This is illustrated in
		Figure~\ref{BWE_k1_pwm_pwm} for BE on the left},
	\begin{figure}[t]
		\centering
		\mbox{\includegraphics[width=0.475\linewidth]{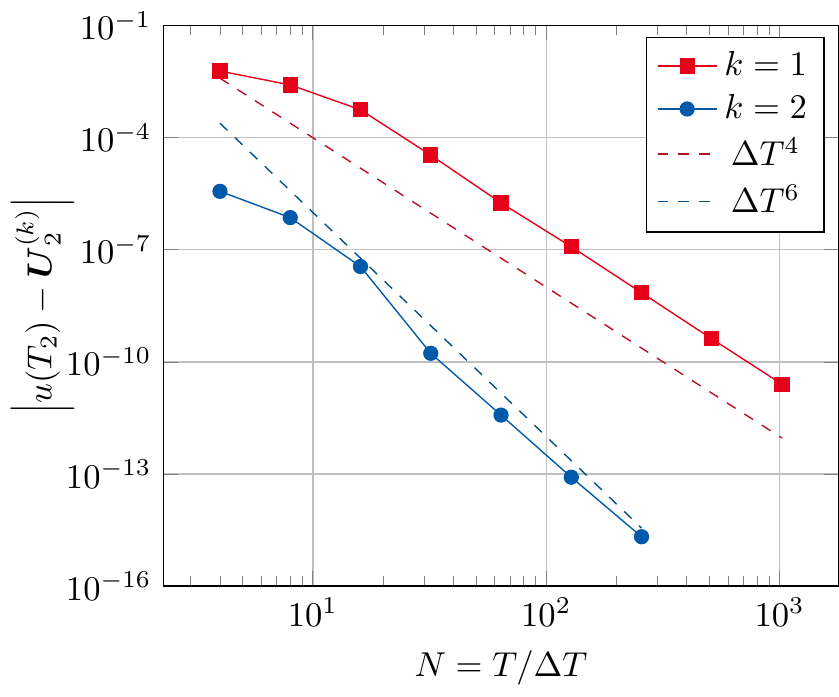}
			\hspace*{0.02\linewidth}
			\includegraphics[width=0.475\linewidth]{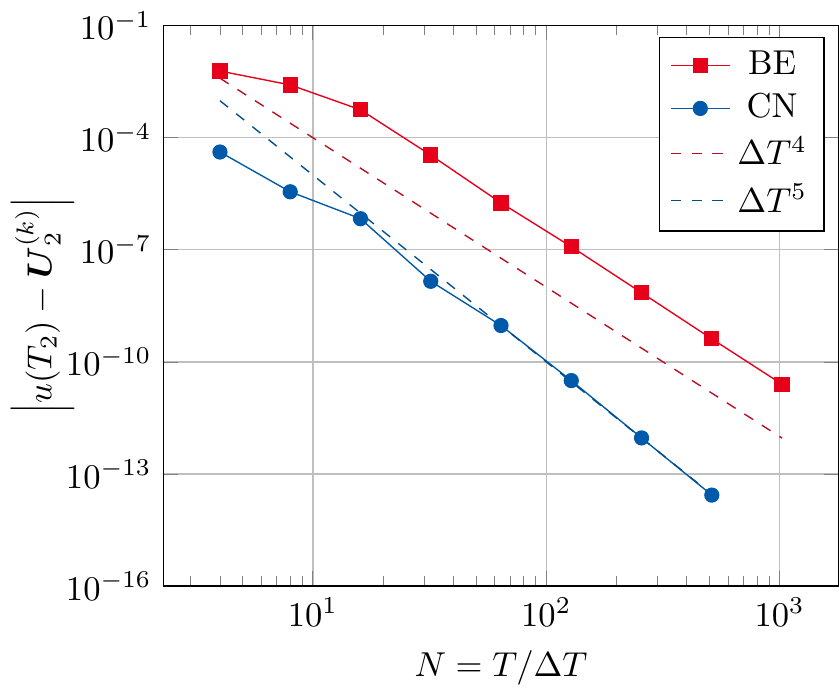}}
		\caption{{Dependence on $N$ of the convergence of the original
				Parareal algorithm. Left: for} $k=1,2$ using {BE, where we
				clearly see order reduction for $N<20$, and the asymptotic
				convergence order is only reached for larger $N$. Right:
				for $k=1$ using CN, where the order
				reduction remains also for larger $N$, in contrast to BE} {Note, BE for $k=1$
				is shown in both plots for reference.}}
		\label{BWE_k1_pwm_pwm}
	\end{figure}
	{where we see that for large $N$ we obtain 4th order convergence
		for $k=1$ and 6th order convergence for $k=2$ which matches the
		prediction $(l+1)(k+1)$ in \eqref{eq:section_Parareal_5} for {BE}
		of order $l=1$. {However,}  for small $N$ (less than 20), the convergence
		order is much lower. On the right in Figure~\ref{BWE_k1_pwm_pwm}, we
		show the corresponding results for the}
		Crank-Nicolson (CN) scheme, {which is of order $l=2$, and we
		iterate only once, $k=1$. Here we observe order reduction to order
		$5$ instead of the predicted order $6$ for smooth input, even for
		larger $N$.}

	\subsection{{Performance of the new Parareal algorithm}}
	
	{We now test our new Parareal algorithm using two} choices of
	input for the coarse propagator with reduced dynamics. On the one
	hand, one could make {the} naive choice of a step function
	\begin{equation}\label{input_step}
	\bar{f}_\mathrm{step}(t)=
	\begin{cases}
	1,\ &t\in[0,T/2),\\
	-1,\ &t\in[T/2,T]
	\end{cases} 
	\end{equation}
	on $[0,T]$. {This} is not globally smooth but piecewise, {which}
	suffices, {since} we consider in the following experiments only
	single step time stepping methods that restart at $T/2$.
	
	On the other hand, in power engineering, the PWM is commonly used as a
	cheap surrogate for sinusoidal excitation. Therefore, its first and
	dominant harmonic, i.e., the sine wave
	\begin{equation}\label{input_sine}
	\bar{f}_\mathrm{sin}(t)=\sin\left(\frac{2\pi}{T}t\right),\quad t\in[0,T],
	\end{equation}
	is {a more} reasonable choice for the coarse problem. The IVP with
	reduced dynamics for {our model problem} is defined by
	\begin{equation}\label{RL_circuit_model_reduced}
	\begin{aligned} 
	\frac{1}{R}\phi^\prime(t)+\frac{1}{L}\phi(t)&=\bar{f}(t),\quad t\in(0,T],
	\\
	{\phi}(0) &= 0
	\end{aligned} 
	\end{equation}
	with $\bar{f}$ being one of the functions in \eqref{input_step} or
	\eqref{input_sine}. The coarse propagator $\bar{\mathcal{G}}$ will
	solve the problem \eqref{RL_circuit_model_reduced}, while the fine
	propagator $\mathcal{F}$ {will solve the original problem}
	\eqref{RL_circuit_model}. The non-smooth part of the input is then
	given by
	\begin{equation} \label{pertubr_pwm_sine}
	\tilde{\f}_{\numteeth}(t):=\pwm_{\numteeth{}}(t)-\bar{f}(t).
	\end{equation}%
	Clearly, $|\tilde{\f}_{\numteeth}(t)|\in
	L^\infty(0,T)${, and Corollary \ref{cor_estimate} gives us the
		error estimate for our new Parareal algorithm}
	\eqref{eq:section_Parareal_3_b_init_reduced}-\eqref{eq:section_Parareal_3_b_reduced}
	in this case.
	
	{We show in Figure~\ref{fig:BWE_k1_pwm_sin} a comparison of the
		convergence behavior of the new Parareal algorithm using BE for
		$k=1$ and $k=2$ iterations using the two different choices of
		reduced input dynamics.}
	\begin{figure}[t]
		\centering
		\mbox{\includegraphics[width=0.475\linewidth]{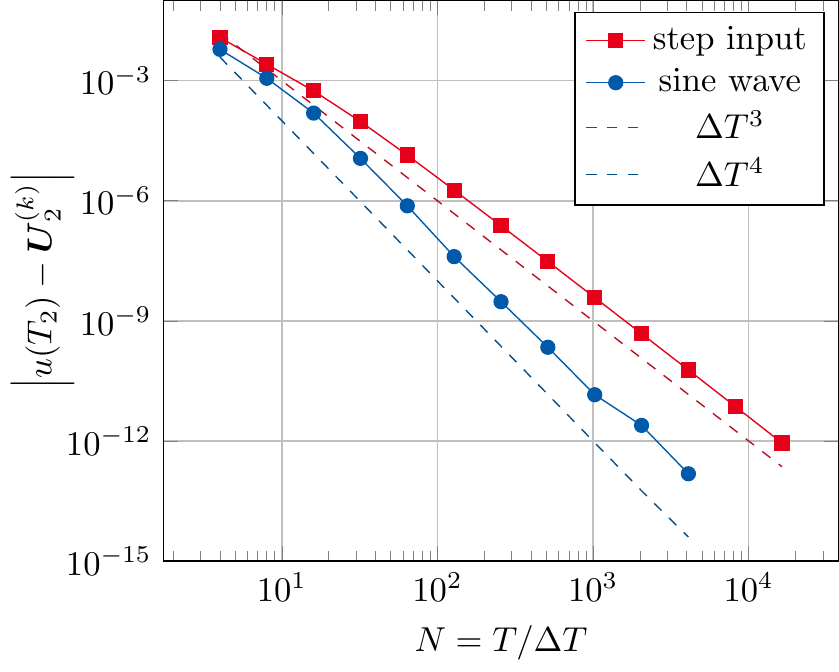}
			\hspace*{0.02\linewidth}
			\includegraphics[width=0.475\linewidth]{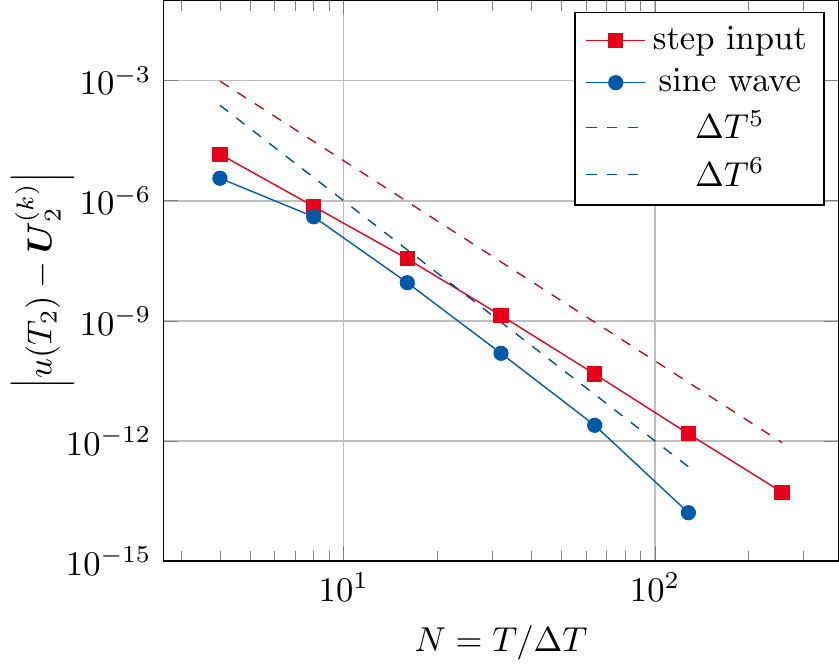}}
		\caption{{Dependence on $N$ of the convergence of the new
				Parareal algorithm using BE and the coarse propagators with
				reduced dynamics \eqref{input_step} and
				\eqref{input_sine}. Left: for $k=1$. Right: for $k=2$.}}
		\label{fig:BWE_k1_pwm_sin}
	\end{figure}%
	{We see that in both cases when the reduced dynamics {of} the step
		function $\bar{f}_\mathrm{step}$ in \eqref{input_step} is used for
		the coarse propagator, one obtains an order reduction{:} for $k=1$ we
		get third order, and for $k=2$ we get fifth order, which} matches
	the theoretical predictions because the lower order term in
	\eqref{eq:section_Parareal_5_b_infty} has order $(l+1)k+1=3$ {for
		$k=1$ and $(l+1)k+1=5$ for $k=2$}. On the other hand, convergence of
	order $(l+1)(k+1)=4$ {for $k=1$ (left) and $(l+1)(k+1)=6$ for $k=2$
		(right)} is observed for the {coarse} sine input $\bar{f}_\mathrm{sin}$, given
	in \eqref{input_sine}{, which means that indeed} the second term
	$\bar{C_3}\left(\Delta T^{l+1}\right)^{k+1}$ in {our}
	estimate~\eqref{eq:section_Parareal_5_b_infty} is dominant over the
	first one. Hence, the sinusoidal function appears to be a well-chosen
	reduced dynamics {for} the coarse problem, which does not slow down
	the convergence of the Parareal algorithm, as the bound in
	\eqref{eq:section_Parareal_5} gives the same rate.
	
	{We next test CN with our new Parareal algorithm. For one
		iteration, $k=1$, we show in Figure~\ref{fig:CN_k1_pwm_sin}}
	{how in this case the step input function $\bar{f}_\mathrm{step}$
		also gives order reduction, we only observe 4th order convergence,
		which is in good agreement with our convergence estimate since} the
	first term in \eqref{eq:section_Parareal_5_b_infty} is of order
	$(l+1)k+1=4$, whereas with the sine input function
	$\bar{f}_\mathrm{sin}$ we get as expected the full 6th order
	convergence.
	\begin{figure}[t]
		\centering
		\includegraphics[width=0.5\linewidth]{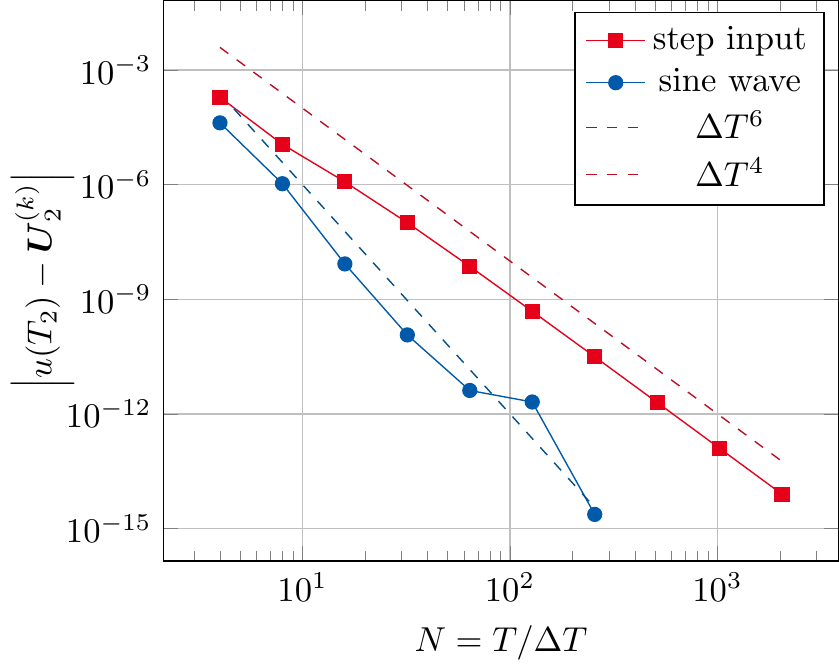}
		\caption{{Dependence on $N$ of the convergence of the new
				Parareal algorithm using CN and the coarse propagators with
				reduced dynamics \eqref{input_step} and \eqref{input_sine} for
				$k=1$.}}
		\label{fig:CN_k1_pwm_sin}
	\end{figure}
	
	\section{Application to an induction machine}\label{Sec:Application}
	
	Due to the low-frequency operating regime of electrical machines,
	their simulation is usually performed assuming that {the} displacement
	current density is negligible with respect to the other current
	densities \cite{Schmidt_2008aa}{, and one} derives a
	parabolic-elliptic initial-boundary value problem from Maxwell's
	equations \cite{Jackson_1998aa}. {This} is called the eddy current
	problem and it reads in terms of the magnetic vector potential
	$\vec{A}:\Omega\times\mathcal{I}\to\mathbb{R}^3$
	\begin{align}
	\sigma\partial_t{\vec{A}}(\vec{r},t)
	+
	\nabla \times \bigl(\nu\nabla\times\vec{A}(\vec{r},t)\bigr) &= \vec{J}{_{\text{src}}}(\vec{r},t) &&\text{{in}}\ \Omega\times\mathcal{I},\label{eq:mqs1}    \\
	\vec{n}\times\vec{A}|_\Gamma &= 0 &&\text{{on}}\ \Gamma{\times\mathcal{I}},\label{eq:mqs1_bdry} \\
	\vec{A}(\vec{r},t_0)&=\vec{A}_{0}(\vec{r}), &&\vec{r}\in\Omega{,}\label{eq:mqs1_init}
	\end{align}
	where $\Omega$ represents the spatial domain of the machine,
	{consisting of a rotor, a stator,} and the air gap in between,
	{$\Gamma=\partial\Omega$ denotes its boundaries,} and
	$\mathcal{I}:=(t_0,t_{\text{end}}]$ is the time interval. The geometry
	is encoded {in} the scalar-valued electric conductivity
	$\sigma=\sigma(\vec{r})\geq0$ and the magnetic reluctivity
	$\nu=\nu(\vec{r},\|\nabla\times\vec{A}\|)>0$. The source current
	density
	$$\vec{J}{_{\text{src}}}=\sum\limits_{{{s}}=1}^{{n_\mathrm{src}}}\vec{\chi}_{{s}}i_{{s}}$$
	impresses lumped currents due to an attached electric network in
	terms of the winding functions
	$\vec{\chi}_{{s}}:\Omega\to\mathbb{R}^3$ which homogeneously
	distribute the currents $i_{{s}}:\mathcal{I}\to\mathbb{R}$ among
	{$n_\mathrm{src}=3$} stranded conductors \cite{Schops_2013aa}, since we
	deal with a three-phase excitation within this application. The
	electric circuit establishes a relation between the current $i_{{s}}$
	and the voltage
	\begin{equation}\label{windFunc}
	v_{{s}}(t)=R_{{s}} i_{{s}}(t)+\int_\Omega \vec{\chi}_{{s}}(\vec{r})\cdot\partial_t{\vec{A}(\vec{r},t)}\;\mathrm{d}\Omega, %
	\end{equation}
	with ${s}=1,2,3$ and $R_{{s}}$ denoting the {direct current} (DC) resistance of the ${{s}}$-th
	stranded conductor.
	
	Furthermore, in order to include {the} rotation of the motor, the
	equation of motion is additionally considered{: the} movement is
	represented in the mesh by the moving band approach
	\cite{Ferreira-da-Luz_2002aa}.  The angular velocity of the rotor,
	\begin{equation}
	\label{eq:rotorAngle}
	\displaystyle\omega(t)=\mathrm{d}_t\theta(t),\quad t\in\mathcal{I},
	\end{equation}
	with a given initial rotor angle $\theta(t_0)=\theta_0$ can be
	determined via
	\begin{align}
	I\mathrm{d}_t\omega+{C\omega}&=T_{\text{mag}}(\vec{A})\quad {{\text{in}\ \mathcal{I}}},\label{eq:motionEqn_omega}\\
	\omega(t_0)&=\omega_0, \label{eq:motionEqn_init}
	\end{align}
	where $I$ is the moment of inertia, and {$C$ is the friction}
	coefficient. System
	\eqref{eq:motionEqn_omega}-\eqref{eq:motionEqn_init} is excited with
	the torque $T_{\text{mag}},$ which is defined on the boundary of the
	air gap.
	
	We consider in the following a {two-dimensional} (2D) computational domain
	$\Omega_\text{2D}\subset\mathbb{R}^2,$ which represents the
	cross-section of the electrical machine. The reduction to the
	2D-setting and discretization of \eqref{eq:mqs1}-\eqref{eq:mqs1_init}
	{using} finite elements with {$n_\text{a}$} degrees of freedom gives
	together with \eqref{windFunc}, \eqref{eq:rotorAngle}, and
	\eqref{eq:motionEqn_omega} an IVP for a coupled system of differential-algebraic equations (DAEs) of the form
	\begin{align}\label{eq:daes3}
	\mathbf{M}\mathrm{d}_t\mathbf{u}(t)
	+
	\mathbf{K}\bigl(\mathbf{u}(t))\mathbf{u}(t)
	&=
	\mathbf{f}(t),\quad {t\in\mathcal{I}},\\
	\mathbf{u}(t_0)&=\mathbf{u}_0,\label{eq:daes3_init}
	\end{align}
	with unknown
	$\mathbf{u}^{\!\top}=[\mathbf{a}^{\!\top},\mathbf{i}^{\!\top},\theta,\omega]:\mathcal{I}\to\mathbb{R}^{n}$
	and the initial condition $\mathbf{u}_0\in\mathbb{R}^{n}$. At each
	point $t$ in time{,} $\mathbf{a}(t)\in\mathbb{R}^{n_\text{a}}$ is the
	vector of (line-integrated) magnetic vector potentials,
	$\mathbf{i}(t)\in\mathbb{R}^{3}$ represents the currents of
	the three phases, $\theta(t)\in\mathbb{R}$ denotes the rotor angle,
	and $\omega(t)\in\mathbb{R}$ is the rotor's angular velocity ($n=n_\text{a}+5$). The
	differential-algebraic nature of the system \eqref{eq:daes3}
	originates from the fact that $\mathbf{M}$ inherits the singularity
	from the finite element conductivity matrix due to the presence of
	non-{conducting} materials in the domain, i.e., where $\sigma=0$. The
	right-hand side $\mathbf{f}(t)$ consists of given voltages
	$\mathbf{v}(t)\in\mathbb{R}^3$ and the mechanical excitation. We refer
	to \cite{Gyselinck_2001aa} for details. Finally, the time-dependent
	problem \eqref{eq:daes3}-\eqref{eq:daes3_init} {has} to be solved via
	application of a time integrator.
	
	\begin{rem}
		We note that the differential-algebraic equation \eqref{eq:daes3} is
		not covered by {our analysis} for ordinary differential
		equations, but we do not expect difficulties for the case of index-1
		problems due the reasoning in \cite{Schops_2018aa}.
	\end{rem}
	
	\subsection{Numerical model}
	\begin{figure}[t]
		\centering
		\includegraphics[width=0.5\linewidth]{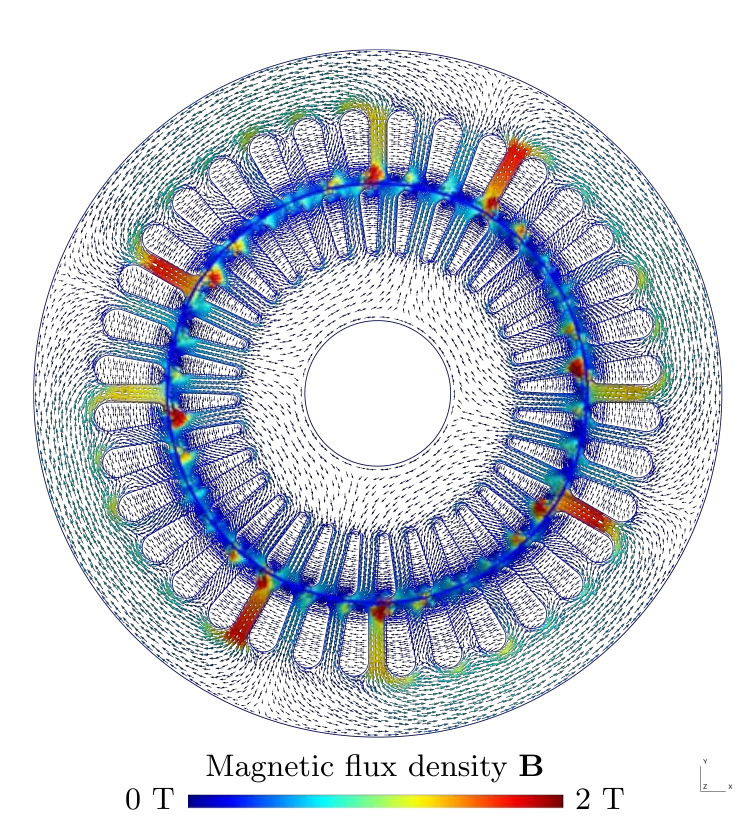}
		\caption{Magnetic field of the four-pole induction machine model
			'im\_3kw' \cite{Gyselinck_2001aa} at time instant $t=0.02$ s if
			excited by a sinusoidal voltage excitation. The numerical
			simulation with GetDP \cite{Geuzaine_2007aa} considers only a
			quarter of the machine geometry with periodic boundary
			conditions.}
		\label{fig:machine_field_plot}
	\end{figure}
	We will now illustrate {the} performance of our {new} Parareal
	algorithm
	\eqref{eq:section_Parareal_3_b_init_reduced}-\eqref{eq:section_Parareal_3_b_reduced}
	for the semi-discrete eddy current problem
	\eqref{eq:daes3}-\eqref{eq:daes3_init}, supplied with a three-phase
	PWM voltage source.  As a {concrete} example we consider a
	four-pole squirrel-cage induction motor, illustrated in
	Figure~\ref{fig:machine_field_plot}, and carry out the computations under no-load operation
	condition. {The} simulation of the 2D machine model was performed
	using the GetDP library \cite{Geuzaine_2007aa} using {$n=4400$} degrees of freedom.  The machine is
	supplied with a three-phase PWM voltage source of $20$ kHz, which
	corresponds to $m=400$ pulses on the time interval $[0,\;0.02]$
	s{,} and is practically relevant for numerous applications in
	electrical engineering.
	{For} $t\in\mathcal{I}$ and ${s}=1,2,3$, the excitation (in V)
	with $m$ pulses {is given} by
	\begin{equation}
	\label{input_pwm_standard}
	v^{\numteeth}_{{s}}(t)
	=
	\mathrm{sign}\left[\sin\left(\dfrac{2\pi}{T} t+\varphi_{{s}}\right)-b_m(t)\right],
	\end{equation}
	where $\varphi_{{s}}$ denotes one of the three phases $\varphi_1=0,$ $\varphi_2=-2/3\pi,$ $\varphi_3=-4/3\pi,$ and 
	\begin{equation}
	b_m(t)=2\left(\dfrac{\numteeth}{T}t - \left\lfloor\dfrac{\numteeth}{T}t\right\rfloor\right)-1
	\end{equation}
	is determined by the bipolar trailing-edge modulation using a sawtooth carrier \cite{Sun_2012aa}.
	
	We consider $T=0.02$ s to be the electric period, which corresponds to
	{a} frequency of $50$ Hz. As an example of the voltage source, a
	PWM signal $v_1^{100}$ of $5$ kHz (corresponding to $m=100$ pulses on
	$[0,0.02]$ s) is {shown} in Figure~\ref{fig:vltg_pwm}.
	\begin{figure}[t]
		\centering
		\includegraphics[width=0.5\linewidth]{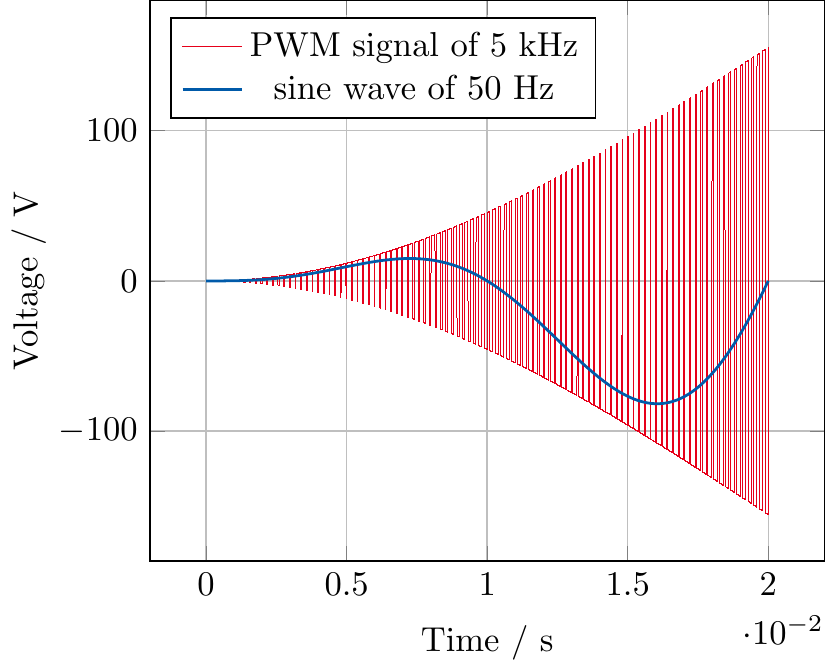}
		\caption{PWM voltage source of $5$ kHz with a {ramp-up} and {phase 1 of the corresponding sinusoidal waveform of $50$ Hz}.}
		\label{fig:vltg_pwm}
	\end{figure}
	An initial {ramp-up} of the applied voltage was {used} for
	reducing the transient behavior of the motor{, as it was proposed}
	by the original authors of the model \cite{Gyselinck_2001aa}. {Phase 1 of the three-phase sinusoidal voltage source of $50$ Hz is shown in Figure~\ref{fig:vltg_pwm}. This waveform will be used as an input for the reduced coarse problem within our new Parareal method.}
	
	The current waveforms, obtained by solving the DAE~\eqref{eq:daes3}
	excited by {the PWM signal of $20$ kHz} are {shown} in
	Figure~\ref{fig:current_pwm}.
	\begin{figure}[t]
		\centering
		\includegraphics[width=0.5\linewidth]{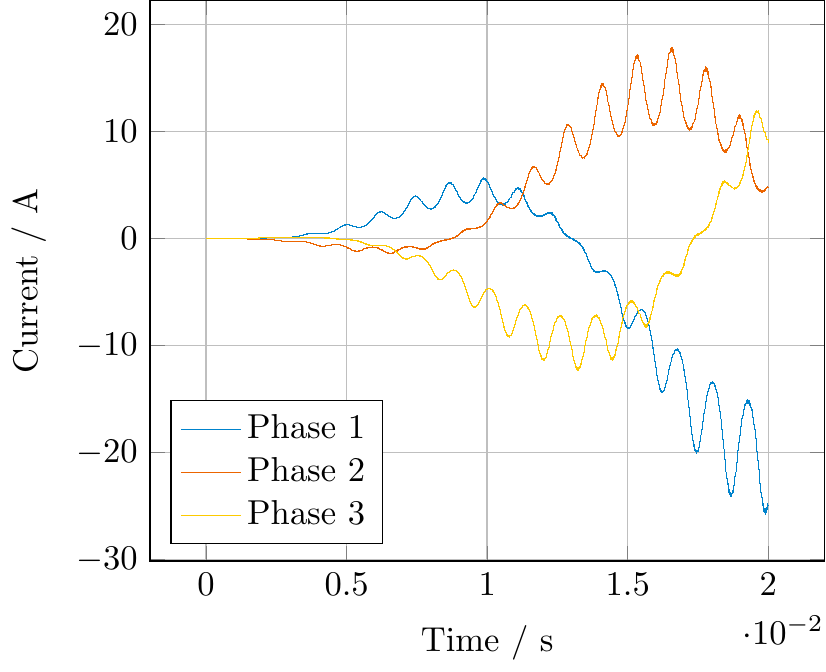}
		\caption{Stator current waveforms for the three-phase PWM voltage source {of $20$ kHz}.}
		\label{fig:current_pwm}
	\end{figure}
	The waveform has multiharmonic characteristics: one observes three
	different time scales: the underlying sinusoidal excitation (50 Hz),
	an additional sinusoidal behavior due to `slotting' of the machine and
	finally the {small} high{-}frequen{cy} oscillations (`ripples') in
	the current waveforms due to the PWM excitations ($20$ kHz). The
	consideration of all those frequencies may be for example important if
	an engineer is concerned with {the} acoustic design of the machine.
	
	\subsection{Application of {the new Parareal algorithm}}
	
	The {new} Parareal algorithm
	\eqref{eq:section_Parareal_3_b_init_reduced}-\eqref{eq:section_Parareal_3_b_reduced}
	was implemented in GNU Octave \cite{Eaton_2015aa} and uses OpenMP
	parallelized calls of GetDP. It was executed on an Intel Xeon cluster
	with $80\times2.00$ GHz cores and 1TB DDR3 memory.
	
	The reduced coarse propagator $\bar{\mathcal{G}}$ solves
	\eqref{eq:daes3}-\eqref{eq:daes3_init} with the input three-phase
	voltage
	\begin{equation}
	\label{input_sin}
	\bar{v}_{{s}}(t)=\sin\left(\dfrac{2\pi}{T} t+\varphi_{{s}}\right),\quad {s}=1,2,3,
	\end{equation}
	for $t\in\mathcal{I}=[0,\;0.02],$ {shown} in
	Figure~\ref{fig:voltage_sin}.
	\begin{figure}[t]
		\centering
		\includegraphics[width=0.5\linewidth]{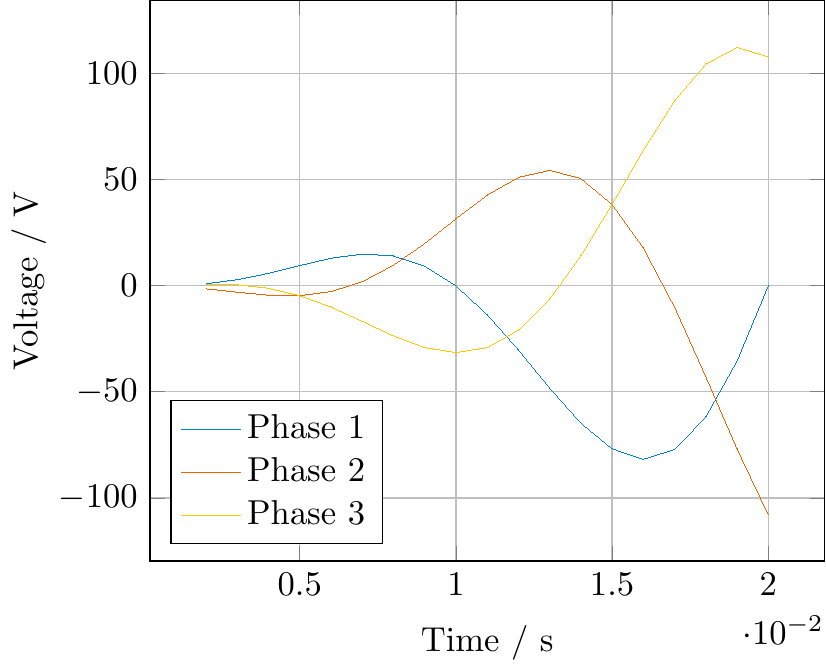}
		\caption{Three-phase sinusoidal voltage source {of $50$ Hz}, used as an input for the
			coarse propagator with reduced dynamics {in our new Parareal algorithm. The c}oarse discretization is
			obtained with the time step $\Delta t=10^{-3}$.}
		\label{fig:voltage_sin}
	\end{figure}
	The fine solver $\mathcal{F}$ uses the original PWM input $v_{{s}}^{400}$,
	${{s}}=1,2,3$ from \eqref{input_pwm_standard}. Both propagators solve the
	IVP {using the Backward Euler} {method} with the time step sizes $\Delta
	t=10^{-3}$ s and $\delta t=10^{-6}$ s for the coarse and the fine
	problem, respectively. We have used $N=20$ cores within the new
	Parareal simulation{, and for comparison we also simulated the
		machine} with the {original Parareal algorithm}
	\eqref{eq:section_Parareal_3_b_init}-\eqref{eq:section_Parareal_3_b}.
	In {the original Parareal algorithm} both coarse and fine problems
	have the same PWM voltage input $v_{{s}}^{400},$ ${s}=1,2,3$ and {we} use
	the {same time} step sizes $\Delta t$ and $\delta t$ defined above.
	
	In order to evaluate the convergence of the {original and the new
		Parareal algorithms, we used} the error norm from \cite[Chapter
	II.4]{Hairer_2000a}, i.e.{,} the vector $\mathbf{u}$ is considered
	close to $\mathbf{v}\in\mathbb{R}^n$ if
	\begin{equation}\label{error_definition}
	\text{err}=\sqrt{\frac{1}{n}\sum_{i=1}^n \frac{|u_i-v_i|^2}{\left(\text{atol} +\text{rtol}\,|v_i|\right)^2}}<1{,}
	\end{equation}
	where $\text{atol}=\text{rtol}=1.5\cdot10^{-5}$ are prescribed
	absolute and relative tolerances. The error norm
	\eqref{error_definition} is applied {to} each jump {at} the $N-1$
	synchronization points. The Parareal iteration is terminated if the
	mismatch of the biggest jump, {measured by
		\eqref{error_definition}}, is below $1$.
	
	The numerical results {in} Figure~\ref{fig:err_parareal_reduced}
	show that the {new Parareal algorithm} with {well-chosen}
	reduced coarse input works very well in practice{: at each
		iteration, it is about one order of magnitude more accurate, and
		also needs in this example 25\% less iterations} than the
	{original} Parareal algorithm {to converge to the desired
		tolerance, capturing well} all relevant frequencies of the
	multiharmonic solution.
	\begin{figure}[t]
		\centering
		\includegraphics[width=0.5\linewidth]{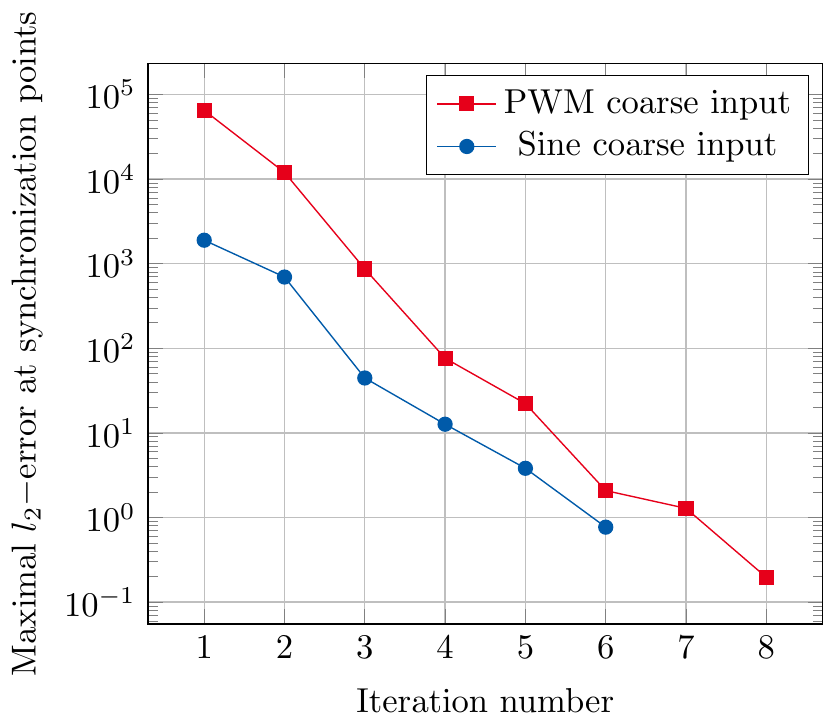}
		\caption{{Comparison of the convergence of the original and
				the new Parareal algorithms}.}
		\label{fig:err_parareal_reduced}
	\end{figure}
	
	\section{Conclusions} \label{section:conclusions}
	
	{{In this paper} we proposed a new Parareal algorithm for} problems
	which involve discontinuous inputs. {Our new Parareal algorithm
		uses }{a} smooth part of the source {representing reduced
		dynamics} for the coarse propagator{, which is} suitable for
	coarse discretization in time. {We analyzed the new Parareal
		algorithm and derived precise error estimates, which show that order
		reduction is possible if the coarse model is not good enough}. In
	particular, if the chosen non-smooth part {$\tilde{\bff}$} of the input belongs to
	$L^p(\mathcal{I},\mathbb{R}^n)$ with $p\geq 1$, and {a} time integrator of order $l$ is
	used as {the} coarse propagator for the reduced problem, {we
		proved that the order reduction can be} at most $l+1/p$. However, if
	the corresponding input with reduced dynamics is a good approximation
	to the original (discontinuous) input, i.e., they are close in the
	sense of the $L^p-$norm, then the {new Parareal algorithm with a
		coarse propagator using large time steps reaches the same order} as
	the {original} Parareal algorithm {that would need a coarse
		propagator with very small time steps. We illustrated the accuracy
		of our estimates with numerical experiments on} an RL-circuit model
	with PWM signal as input{, and we also tested the new Parareal
		algorithm on an} eddy current problem, describing the operation of
	an induction machine. The new Parareal algorithm {is about an order
		of magnitude more accurate in each iteration than the original one,
		and reaches} a prescribed tolerance {in $25\%$ less iterations
		in the eddy current example.}

\end{document}